\documentclass[12pt,a4paper,twoside]{amsart}
\usepackage{a4wide, charter, amsmath,amsthm, amsbsy,amsfonts,amssymb, bm, manfnt}
\usepackage[mathscr]{eucal}
\usepackage{mathrsfs,graphicx, amscd, tikz-cd, cancel}
\usetikzlibrary{matrix,arrows,decorations.pathmorphing}
\usepackage[normalem]{ulem}

\usepackage{amssymb}

\def\CC{{\mathbb C}}

\def\ZZ{{\mathbb Z}}

\def\Ecal{{\mathcal E}}

\def\Ocal{{\mathcal O}}

\def\Pcal{{\mathcal P}}

\def\Xcal{{\mathcal X}}

\def\gfrak{{\mathfrak g}}

\def\At{{\mathcal A t}}
\def\pt{{\scriptscriptstyle\bullet}}

\newcommand\even{\textrm{even}}

\newcommand\at{\operatorname{\mathfrak{at}}}
\newcommand\ad{\operatorname{ad}}
\newcommand\Ad{\operatorname{Ad}}
\newcommand\Aut{\operatorname{Aut}}

\newcommand\Hom{\mathcal{H}om}
\newcommand\HHom{\operatorname{Hom}}

\newcommand\spec{\operatorname{Spec}}
\newcommand\sym{\operatorname{Sym}}
\newcommand\Out{\operatorname{Out}}
\newcommand\PSL{\operatorname{PSL}}

\newcommand\slin{\operatorname{\mathfrak{sl}}}

\newtheorem{theorem}{Theorem}[section]
\newtheorem{proposition}[theorem]{Proposition}
\newtheorem{lemma}[theorem]{Lemma}
\newtheorem{corollary}[theorem]{Corollary}

\theoremstyle{definition}

\newtheorem{remark}[theorem]{Remark}

\numberwithin{equation}{section}

\usepackage[
	hypertexnames=false,
	hyperindex,
	pagebackref,
	pdftex,
	breaklinks=true,
	bookmarks=false,
	colorlinks,
	linkcolor=blue,
	citecolor=red,
	urlcolor=red,
]{hyperref}

\begin{document}
\baselineskip=15.5pt

\title[Characteristic forms for families of local systems]{Characteristic forms for
holomorphic families of local systems}

\author[I. Biswas]{Indranil Biswas}

\address{Department of Mathematics, Shiv Nadar University, NH91, Tehsil Dadri,
Greater Noida, Uttar Pradesh 201314, India}

\email{indranil.biswas@snu.edu.in, indranil29@gmail.com}

\author[E. Looijenga]{Eduard Looijenga}

\address{Mathematisch Instituut Universiteit Utrecht, Netherlands and Dept.\ of Mathematics, University of Chicago, USA
}

\email{e.j.n.looijenga@uu.nl}



\date{}

\begin{abstract}
Let $\Gamma$ be a finitely generated group and $G$ a complex Lie group with Lie algebra $\gfrak$. The complex variety of group 
homomorphisms $\HHom (\Gamma,\, G)$ is equipped with an action of $G$ by inner automorphism and thus it defines a complex-analytic 
stack. We associate to each $\Phi\,\in\, (\sym^n \gfrak)^G$ a closed holomorphic $n$-form on this stack with values in 
$H^n(\Gamma,\, \CC)$. This construction is multiplicative so that we get a graded algebra homomorphism. We obtain this as a universal 
refinement of a Chern-Weil homomorphism.

If we take $n\,=\,2$ and let $\Gamma$ be the fundamental group of a closed orientable surface, then we recover a classical 
construction due to Goldman. 
\end{abstract}

\maketitle

\hfill{\textsl{To Ezra, on the occasion of his 60th birthday.}}

\section*{Introduction}
Let $\Gamma$ be the fundamental group of closed connected orientable surface of genus $>1$, $G$ a Lie group and 
$\sigma$ a nondegenerate bilinear form on its Lie algebra $\gfrak$ that is invariant under the adjoint representation. A well-known theorem of
Goldman \cite{Go} asserts that $\sigma$ determines on the space of group homomorphisms $\Gamma\,\longrightarrow\, G$
given up to $G$-conjugacy (understood in the sense of geometric invariant theory) a symplectic form. One of the main goals of this paper is to show that this is a special case of what one might call a \emph{universal relative Chern-Weil homomorphism}. 

In more detail, let us first recall the Chern-Weil recipe for extracting topological invariants from a principal $G$-bundle 
$\Pcal\,\longrightarrow\, M$ over manifold a $M$. One begins with choosing a connection $\nabla$ on the bundle $\Pcal$ over $M$. Its 
curvature form $R(\nabla)$ is then a $2$-form on $M$ that takes values in an associated bundle of Lie algebras 
$\ad(\Pcal)\,=\,\Pcal\times^G \gfrak$. A fiber of $\ad(\Pcal)$ is isomorphic to the Lie algebra $\gfrak$ of $G$ with 
the isomorphism given up to the adjoint action of $G$ on $\gfrak$ and so it makes sense to evaluate on $R(\nabla)$ a 
homogeneous polynomial function on $\gfrak$ that is invariant under this adjoint action of $G$, the result then being a 
differential form on $M$ of degree twice that of the polynomial. The vector bundle $\ad(\Pcal)$ inherits a connection from $\Pcal$ and the 
Bianchi identity, which says that $R(\nabla)$ is closed with respect to the inherited connection, implies that such a 
differential form is closed. A simple homotopy argument shows that its de Rham cohomology class is independent of the choice 
of $\nabla$ and we thus obtain the Chern-Weil homomorphism $\CC[\gfrak]^G\,\longrightarrow\, H^{\even}(M;\, \CC)$.

The relative version we alluded to begins (in the complex-analytic setting, say) with a proper holomorphic submersion $f\,:\,\Xcal\,\longrightarrow\, S$
of complex manifolds and takes for $G$ a complex linear algebraic group (so with a complex Lie algebra $\gfrak$), a holomorphic principal $G$-bundle
$\Pcal$ over $\Xcal$ and a holomorphic connection $\nabla$ on $\Pcal$ \emph{relative to} $f$ (it is given by a holomorphic
splitting of the relative Atiyah exact sequence, which means that only covariant derivation along the fibers of
$f$ is defined). We assume that $\nabla$ is flat and so this can be thought of 
as a holomorphically varying family of compact complex manifolds endowed with a holomorphic principal 
$G$-bundle with flat connection. We show that a refinement of the Chern-Weil homomorphism 
yields a graded algebra homomorphism 
\[
\Psi_{(f, \Pcal, \nabla)}\, :\, \CC[\gfrak]^G\,\longrightarrow\, \bigoplus_{n\ge 0} 
H^0(S,\,\Omega^n_{S,cl}\otimes R^nf_*\CC),
\]
where $\Omega^n_{S,cl}$ stands for the sheaf
of closed holomorphic $n$-forms on $S$.

Our generalization of Goldman's construction is in a sense a universal homomorphism of this type. It is stated in terms of the 
\emph{character space} that is associated with a finitely generated group $\Gamma$ (that takes the role of the fundamental group of a 
fiber of $f$) and a complex Lie group $G$ as above: this is the complex-analytic space of group homomorphisms $\HHom (\Gamma, \,G)$. It comes 
with an evident action of $\Aut(\Gamma)\times \Aut(G)$.% In particular, $G$ acts on it via inner automorphisms. 
The following theorem 
(stated here only in the holomorphic setting) is, we believe, the most natural generalization of the main result of \cite{Go}.

\begin{theorem}\label{thm:main}
There is a natural graded $\CC$-algebra homomorphism $\Psi_{\Gamma, G}$
from $\CC[\gfrak]^G$ to the algebra of $\Aut(\Gamma)$-invariant sections of $\bigoplus_{n\ge 0}\Omega^n_{\HHom (\Gamma, G), cl}\otimes H^n(\Gamma,\, {\mathbb C})$
that are killed by contraction with the vector fields defined by $\gfrak$. It is universal as a (refined) Chern-Weil homomorphism in the sense that
is explained below. 
\end{theorem}

Let us first comment on the equivariance properties of $\Psi_{\Gamma, G}$. Naturality of $\Psi_{\Gamma, G}$ is here understood as 
contravariance in $G$ and covariance in $\Gamma$. Since $\Aut(\Gamma)$ and the inner automorphisms of $ G$ act trivially on $\CC[\gfrak]^G$, 
this implies that $\Psi_{\Gamma, G}$ takes values in forms that are $\Aut(\Gamma)\times G$-invariant (the $\Gamma\times G$-invariance also easily seen directly). So 
we may just as well say that $\Psi_{\Gamma, G}$ is $\Out(G)$-equivariant and takes values in $\Out(\Gamma)$-invariant forms.

In order to explain its universal nature, it is best to reformulate the first assertion in terms of the \emph{character stack}
$\underline{\chi}_G(\Gamma)$ defined by the $G$-action on $\HHom (\Gamma,\, G)$. A \emph{regular form} on this stack $\underline{\chi}_G(\Gamma)$ is by
definition a regular form on the space of group homomorphisms $\HHom(\Gamma,\, G)$ (which comes with the structure of a 
complex-analytic space) that is \emph{basic} with respect to the $G$-action. This means that both the form and its exterior derivative are killed by $\gfrak$. So in the above theorem we can replace 
$\Omega^n_{\HHom (\Gamma, G)}$ by $\Omega^n_{\underline{\chi}_G(\Gamma)}$ and regard $\Psi_{\Gamma, G}$ as a graded, $\Out(G)$-equivariant algebra homomorphism 
\[
\underline \Psi_{\Gamma, G}\, :\,\CC[\gfrak]^G\,\longrightarrow\, \big(\bigoplus_{n\ge 0}H^0(\underline\chi_G(\Gamma),\,\Omega^n_{\underline\chi_G(\Gamma), cl})
\otimes H^n(\Gamma,\, {\mathbb C})\big)^{\Out(\Gamma)}.
\]
Now suppose that we are given a triple $(f :\Xcal\to S, \,\Pcal,\, \nabla)$ as above with connected fibers. Since our interest is focused on
a property that is local on $S$, we shall assume that $S$ is simply connected and $f$ admits a section $x\,:\,S\,\longrightarrow\, \Xcal$. This implies that for each $s\in S$, the inclusion map $X_s\,\subset\,\Xcal$ induces an isomorphism on fundamental groups. Let us write $\Gamma$ for $\pi_1(\Xcal,\, x)$, where we think of $x$ as a (relative) base point. 
We then have a natural homotopy class of maps $\kappa$ from $\Xcal$ to a classifying space $B\Gamma$ of $\Gamma$. The associated map
$$\kappa^*\,:\, H^\pt (\Gamma,\,\CC)\,\cong\, H^\pt (B\Gamma,\,\CC)\,\longrightarrow\, H^\pt(\Xcal, \,\CC)\,
\longrightarrow\, H^0(S, \,R^nf_*\CC)$$ is independent of $x$. Each fiber $X_s$ carries a flat principal $G$-bundle. So if we choose a lift of $x$ to $\Pcal$, then the monodromy representations define
a morphism $\rho\,: \,S\,\longrightarrow\, \HHom (\Gamma, G)$. But if we pass to the stack, then the resulting morphism
$\underline\rho\,:\, S\,\longrightarrow\, \underline{\chi}_G(\Gamma)$ no longer depends on $x$ and its lift to $\Pcal$. The last sentence of Theorem \ref{thm:main} is meant to say that $\Psi_{(f, \Pcal, \nabla)}$ is equal to the composite 
\[
\CC[\gfrak]^G\,\xrightarrow{\,\underline\Psi_{\Gamma, G}\,}\, \bigoplus_{n\ge 0}H^0(S,\,\Omega^n_{\underline{\chi}_G(\Gamma), cl})\otimes H^n(\Gamma,\,\CC)
\,\xrightarrow{\,\underline\rho^*\otimes \kappa^*\,} \,\bigoplus_{n\ge 0}H^0(S,\,\Omega^n_{S,cl}\otimes R^nf_*\CC).
\]
This makes it apparent that the
cohomology of the fibers only involves classes coming from their fundamental group (as is to be expected).

Goldman is concerned with $G$-invariant quadratic forms $B$ on $\gfrak$ (so $n\,=\,2$) and takes for $\Gamma$ the fundamental group of a closed orientable genus $g$ 
surface. There is no group cohomology in his formulation, as implicit in his discussion is an identification $H^2(\Gamma)\cong \ZZ$. It should however not be left 
out if one demands (as he does) $\Out(\Gamma)$-invariance. Indeed, an orientation reversing automorphism of the closed surface induces minus the identity in 
$H^2(\Gamma)$. Our proof, when restricted to this special case seems different from his. We should however add that Goldman also shows that if $B$ is nondegenerate, 
then so is the associated closed $2$-form on $\underline{\chi}_G(\Gamma)$, so that the latter becomes what one might call a \emph{holomorphically symplectic stack} 
(he works in the differentiable category and avoids stacks, however this difference is merely formal).

We note here that Guruprasad-Huebschmann-Jeffrey-Weinstein \cite{GHJW} obtained an extension of Goldman's theorem 
to a situation where the fibers of $f$ are compact (oriented) surfaces with boundary (but where the $G$-invariant is still 
a quadratic form on its Lie algebra) and where one might want to fix the holonomy on each boundary component. Here 
the universal situation not only involves the fundamental group of the surface, but also the fundamental groups of 
the boundary components and the way these groups are related in a groupoid fashion. In our set-up, this means that instead
of requiring $f$ to be proper, we only demand that it be topologically locally trivial. It is likely that our 
approach lends itself to an extension which incorporates that case (we note that the formula (\ref{eqn:ATrel}) only requires $f$
to be smooth) but we have not pursued this.

There is presumably also a link with work of 
Bloch and Esnault. Given a family of connections with certain properties, there are secondary cohomology
invariants associated with it which are known as Cheeger--Chern--Simons classes \cite{CherS}, \cite{CheeS}. In \cite{BE}, refinements
of these are defined in an algebro-geometric set-up (see
\cite[Theorem 3.7]{BE} and \cite[Definition 3.10]{BE}).
Their constructions use the Atiyah class (which is also the key ingredient here), but their invariants take values in a different complex of forms: their focus is on vector bundles rather than on principal bundles (in other words, $G$ is then a general linear group) and their invariants live on the associated projectivized bundle. It is at present not clear to us whether
in this case the closed (group cohomology valued) forms that we find can be extracted from theirs.
\\

Goldman's theorem was tailored to be applied to the moduli space of projective structures on a fixed closed connected orientable
surface $\Sigma$ of genus $g\,\ge \,2$ (these are projective structures given up to a diffeomorphism of $\Sigma$ isotopic to the identity map) and indeed, this is what he did in the same paper. He thus obtained on that moduli space a natural holomorphic symplectic
form. A projective structure on $\Sigma$ determines a complex-analytic structure and hence this moduli space maps to the Teichm\"uller 
space of $\Sigma$ in a manner that is equivariant for the action of the mapping class group of $\Sigma$. It is however not clear a 
priori that this can be done is a purely algebraic setting, where we work with schemes (or Deligne-Mumford stacks) of finite type over 
an algebraically closed field of characteristic zero. At the outset one needs an algebraic definition of a projective structure. 
Indeed, to do this as well as to show that it leads to symplectic form was our original motivation. In \cite{BL}
we show that our set-up allows us to accomplish this in a straightforward manner.
\\

We happily dedicate this paper to Ezra Getzler on the occasion of his 60th birthday. 

\section{Relative versions of the Chern-Weil homomorphism}\label{sect:charforms}

\subsection{Chern-Weil homomorphism for a relative connection}

We fix a complex Lie group $G$. Let $f\,:\,\Xcal\,\longrightarrow\, S$ be a
holomorphic submersion of complex 
manifolds, and let $\pi\,:\, \Pcal\,\longrightarrow\, \Xcal$ be a holomorphic principal $G$-bundle. We assume that 
$\Pcal$ is endowed with a \emph{holomorphic connection $\nabla$ relative to $f$}. Let us first spell out what this 
exactly means. First take the direct image under $\pi$ of the short exact sequence $$0\,\longrightarrow\, 
\theta_{\Pcal/\Xcal}\,\longrightarrow\, \theta_\Pcal\,\longrightarrow\, \pi^*\theta_\Xcal\,\longrightarrow\, 0$$ of the
sheaf of holomorphic vector fields; it is still exact and that remains so if we subsequently pass to $G$-invariants. So 
letting $\ad(\Pcal)$ stand for the $\Ocal_\Xcal$-module of Lie algebras $\pi_*\theta_{\Pcal/\Xcal}^G$, we have a 
short exact sequence on $\Xcal$
\begin{equation}\label{eqn:adsequence}
0\,\longrightarrow\, \ad(\Pcal)\,\longrightarrow\, \pi_*(\theta_\Pcal)^G\,\longrightarrow\, \theta_\Xcal \,\longrightarrow\, 0.
\end{equation}
A \emph{holomorphic connection $\nabla$ on $\Pcal$ relative to $f$} is by definition a holomorphic splitting of this sequence over the subsheaf 
$\theta_{\Xcal/S}\,\subset\, \theta_\Xcal$. One might picture the triple 
$(\Xcal/S,\,\Pcal,\, \nabla)$ as a holomorphic family of compact complex manifolds, each of which is endowed with a holomorphic principal
$G$-bundle together with a holomorphic
connection on it, that is parametrized by $S$. Given these data, consider the commutative diagram of short exact sequences, one of them 
coming with a section
\begin{equation}\label{c1}
\begin{tikzcd}
 & & 0\arrow[d] &0\arrow[d] &\\
0\arrow{r}& \ad(\Pcal)\arrow[r]\arrow[d, equal] & \pi_*\theta_{\Pcal/S}\arrow{r}\arrow[d] & \theta_{\Xcal/S}\arrow{r}
\arrow[d]\arrow[bend right]{l}[description]{\nabla}&0\\
0\arrow{r}& \ad(\Pcal)\arrow[r] & \pi_*\theta_{\Pcal}\arrow{r}\arrow[d] & \theta_\Xcal\arrow{d}\arrow{r} &0\\
 & & f^*\theta_S\arrow[r, equal]\arrow[d] & f^*\theta_S \arrow[d] & \\
 & & 0 &0 &\\
\end{tikzcd}
\end{equation}
The relative holomorphic connection $\nabla$ lifts locally on $\Xcal$ to an absolute holomorphic connection
(i.e., lifts to local holomorphic homomorphisms from $\theta_\Xcal$ to $\pi_*\theta_{\Pcal}$). The obstruction against being able to do this
globally defines the Atiyah class 
\begin{equation}\label{u1}
\at_{\Xcal}(\Pcal)\,\in\, H^1(\Xcal, \,\Hom(\theta_\Xcal,\,\ad(\Pcal))\,\cong\, H^1(\Xcal,\, \Omega_\Xcal\otimes\ad(\Pcal)).
\end{equation}
It is by definition the extension class 
defined by the second row in \eqref{c1}. If we work relative to $S$, then this evidently produces an element
$\at_{\Xcal/S}(\Pcal)\,\in\, R^1f_*(\Omega_\Xcal\otimes\ad(\Pcal))$. But then the given relative connection $\nabla$ lifts this to 
\begin{equation}\label{at1}
\at^\nabla_{\Xcal/S}(\Pcal)\,\in\, H^0(S,\, R^1f_*(f^{-1}\Omega_S\otimes_{f^{-1}\Ocal_S}\ad(\Pcal)))
\,\cong\, H^0(S,\,\Omega_S\otimes R^1f_*\ad(\Pcal)).
\end{equation}
Indeed, the cokernel of $\theta_{\Xcal/S}\,\xrightarrow{\,\nabla\,}\, \pi_*\theta_{\Pcal/S}\,\subset\, \pi_*\theta_{\Pcal}$ is an 
extension of $f^*\theta_S$ by $\ad(\Pcal)$ and $\at^\nabla_{\Xcal/S}(\Pcal)$
in \eqref{at1} is its extension class. This \emph{relative Atiyah class} 
in \eqref{at1} can be explicated as follows. Write $\Ecal^{p,q}_\Xcal$ for the (flabby) sheaf of $C^\infty$-forms on $\Xcal$ of type $(p,\,q)$. The 
relative connection $\nabla$ lifts locally on $\Xcal$ to an absolute one and a partition of unity then allows us do this globally with 
$C^\infty$ coefficients so that we obtain a $C^\infty$ connection
\[
\widetilde\nabla \,:\, \theta_\Xcal\, \longrightarrow\,\pi_*\theta_{\Pcal}.
\]
Its curvature form $R(\widetilde\nabla)$ will be a global section of 
\[
(\Ecal^{2,0}_\Xcal\oplus \Ecal^{1,1}_\Xcal) \otimes \ad(\Pcal)\,=\,
(\Ecal^{1,0}_\Xcal \oplus \Ecal^{0,1}_\Xcal)\wedge\Omega^1_\Xcal \otimes \ad(\Pcal).
\]
Its restriction to every fiber of $f$ is holomorphic and hence the component of type $(1,\,1)$ must be
a section of the subsheaf $\Ecal^{0,1}_\Xcal\wedge f^*\Omega^1_S \otimes \ad(\Pcal)\,\subset\,
\Ecal^{0,1}_\Xcal\wedge\Omega^1_\Xcal \otimes \ad(\Pcal)$.
The Bianchi identity says that $R(\widetilde\nabla)$ is $\widetilde\nabla$-closed and this implies (see for
example \cite{MS}) that the above component of type $(1,\,1)$ must be
$\overline\partial$-closed. It thus defines a section of 
$$R^1f_* (f^*\Omega^1_S \otimes \ad(\Pcal))\,\cong\, \Omega_S\otimes R^1f_*(\ad(\Pcal))$$
and this is indeed our $\at^\nabla_{\Xcal/S}(\Pcal)$.

\subsubsection*{Associated Chern-Weil homomorphism}

Denote by $\gfrak$ the Lie algebra of $G$, and let $$\Phi\,\in\, \CC[\gfrak]^G$$ be a 
$G$-invariant homogeneous polynomial of degree $n$. We can evaluate $\Phi$ on $\at_{\Xcal}(\Pcal)$ in 
\eqref{u1} and with the help of the cup product obtain an element of 
$H^n(\Xcal,\,\Omega^n_\Xcal)$. When $\Xcal$ is compact K\"ahler, we have 
$H^n(\Xcal,\,\Omega^n_\Xcal)\,\cong\, H^{n,n}(\Xcal)$ and according to Atiyah \cite{At} 
this reproduces the cohomology class given by the image of $\Phi$ under the Chern-Weil homomorphism associated to 
$\Pcal$.

The given relative connection $\nabla$ allows us to construct the following
refinement of the above construction: The evaluation of 
$\Phi$ on $\at^\nabla_{\Xcal/S}(\Pcal)$ yields $$\Phi(\at^\nabla_{\Xcal/S}(\Pcal))\,\in\, H^0(\Omega^n_S\otimes 
R^nf_*\Ocal_\Xcal).$$ We thus obtain a homomorphism of graded algebras
\begin{equation}\label{eqn:ATrel}
\CC[\gfrak]^G\,\longrightarrow\, \bigoplus_{n \geq 0} H^0(S,\, \Omega^n_S\otimes R^nf_*\Ocal_{\Xcal}).
\end{equation}

If we assume that $f$ is proper, then by the Grauert direct image theorem \cite{Gr}, each $R^nf_*\Ocal_\Xcal$ is a coherent $\Ocal_S$-module. The 
relative Chern-Weil homomorphism then commutes with base change: if $g\,:\,S'\,\longrightarrow\, S$ is a holomorphic map of complex 
manifolds and $$\widetilde{g}\,:\,\Xcal'\,:=\,\Xcal\times_SS'\,\longrightarrow\,\Xcal$$ is the pulled back
family, so that we have the pull-back
$(\widetilde{g}^*\Pcal,\, \widetilde{g}^*\nabla)$ of 
$(\Pcal, \,\nabla)$ over $\Xcal'/S'$, then $$\Phi(\at_{\Xcal'/S'}^{\widetilde{g}^*\nabla})\,=\,g^*\Phi(\at_{g^*\Xcal/S'}^\nabla).$$

If $\Xcal$ is even compact and admits a K\"ahler metric, then by a classical result of Deligne \cite{deligne:ss} the Leray spectral 
sequence for $f$ applied to the constant sheaf $\underline{\CC}_\Xcal$ degenerates. If we combine this with the Hodge decomposition theorem we 
find that $H^0(S,\Omega^n_S\otimes R^nf_*\Ocal_\Xcal)$ naturally sits inside the $n$-th Leray subquotient of $H^{2n}(\Xcal;\, \underline{\CC})$ and that 
this is also the image of the Chern-Weil class of $\Pcal$ in $H^{2n}(\Xcal;\, \underline{\CC})$ defined by $\Phi$.

\subsection{Chern-Weil homomorphism for a flat relative connection}\label{subsect:relflat}

\emph{We here assume that $f$ is proper and that $\nabla$ is \emph{flat} as a relative holomorphic connection.} By 
the last property we mean that the associated section of $\pi_*\theta_{\Pcal/S}\,\longrightarrow\, \theta_{\Xcal/S}$ 
preserves the Lie bracket of vector fields. This is equivalent to the curvature form $R(\nabla)$ (which is here a section of 
$\Omega^2_{\Xcal/S}\otimes \ad(\Pcal)$) being zero. The homomorphism $\ad(\Pcal)\,\longrightarrow\, 
\Omega_{\Xcal/S}\otimes\ad(\Pcal)$ given by $\nabla$ will, for notational convenience, again be denoted by $\nabla$. 
The kernel of the map $\nabla\,:\, \ad(\Pcal)\,\longrightarrow\, \Omega_{\Xcal/S}\otimes\ad(\Pcal)$ is then a 
subsheaf $$\ad(\Pcal)^\nabla\,\subset\, \ad(\Pcal)$$ of locally free $f^{-1}\Ocal_S$-modules with the property that 
$$\Ocal_\Xcal\otimes_{f^{-1}\Ocal_S} \ad(\Pcal)^\nabla\,\longrightarrow\,\ad(\Pcal)$$ is an isomorphism of 
$\Ocal_\Xcal$-modules. One might call this then a \emph{local system relative to} $f$ and regard 
$(\Xcal/S,\,\Pcal,\, \nabla)$ as a family of compact complex manifolds, each of which is endowed with a holomorphic principal 
$G$-bundle and a holomorphic flat connection on it. The higher direct images $R^nf_*\ad(\Pcal)^\nabla$ can be computed with a 
relative de Rham resolution. Since $f$ is proper, it then follows that these are in fact coherent $\Ocal_S$-modules.

\begin{lemma}\label{lemma:Atrel}
In the above set-up the class $\at^\nabla_{\Xcal/S}(\Pcal)\,\in\, H^0(S,\, \Omega_S\otimes R^1f_*\ad(\Pcal))$ has as a 
natural lift $$\At^\nabla_{\Xcal/S}(\Pcal)\,\in\, H^0(S,\,\Omega_S\otimes R^1f_*\ad(\Pcal)^\nabla).$$ Moreover, if 
$\widetilde\nabla$ is a $C^\infty$-lift of $\nabla$ to an absolute connection on $\ad(\Pcal)$, then 
$\widetilde\nabla$ induces a connection on $f_*\ad(\Pcal)^\nabla$ as well as on the higher direct images, and 
$\At^\nabla_{\Xcal/S}(\Pcal)$ is closed for the resulting connection on $R^1f_*\ad(\Pcal)^\nabla$.
\end{lemma}

\begin{proof}
If we go through the preceding procedure, we note that $R(\widetilde\nabla)$ has now the additional property that
its restriction to every fiber of $f$ is zero. This means that $R(\widetilde\nabla)$ is a section of
\begin{equation}\label{at2}
\Ecal^1_\Xcal\wedge f^*\Omega_S\otimes \ad(\Pcal)\,\,=\,\, \Ecal^1_\Xcal\wedge_{f^{-1}\Ocal_S} \big(f^{-1}\Omega_S\otimes_{f^{-1}\Ocal_S} \ad(\Pcal)^\nabla\big),
\end{equation}
where $\Ecal^1_\Xcal\,=\,\Ecal^{1,0}_\Xcal \oplus \Ecal^{0,1}_\Xcal$. The Bianchi identity says that $R(\widetilde\nabla)$ is 
$\widetilde\nabla$-closed. Since on the right hand side of \eqref{at2} we have already passed to the relative local system $\ad(\Pcal)^\nabla$, this 
implies that its image in $\Ecal^1_{\Xcal/S}\otimes_{f^{-1}\Ocal_S} \big(f^{-1}\Omega_S\otimes_{f^{-1}\Ocal_S} \ad(\Pcal)^\nabla\big)$ is 
$d_{\Xcal/S}$-closed. Hence we now obtain a section $\At^\nabla_{\Xcal/S}(\Pcal)$ of $R^1f_*(f^{-1}\Omega_S\otimes_{f^{-1}\Ocal_S} 
\ad(\Pcal)^\nabla)=\Omega_S\otimes R^1f_*\ad(\Pcal)^\nabla$.

The direct image $R^1f_*\ad(\Pcal)^\nabla$ inherits from $\widetilde\nabla$ a connection: the covariant derivative 
with respect to a local $C^\infty$ vector field on $S$ is computed by choosing a lift of that vector field to $\Xcal$ and then
taking the covariant derivative with respect to it for the connection $\widetilde\nabla$. The induced action on the higher direct 
images of $\ad(\Pcal)$ is independent of the lift of the vector field because two such lifts differ by a vector field along the fibers 
of $f$ and covariant derivation with respect to such a vector field acts trivially on $R^\pt f_*\ad(\Pcal)^\nabla$. 
Since $\At^\nabla_{\Xcal/S}(\Pcal)$ represents the curvature form, the Bianchi identity implies that 
$\At^\nabla_{\Xcal/S}(\Pcal)$ is closed relative to the above connection on $R^1 f_*\ad(\Pcal)^\nabla$.
\end{proof}

We can understand $\At^\nabla_{\Xcal/S}(\Pcal)$ as a \emph{Kodaira-Spencer class}: if we regard $(\Xcal/S,\,\Pcal,\, 
\nabla)$ as a family of holomorphic principal $G$-bundles with flat connections, then it is the first obstruction against this 
family being locally trivial over $S$. In this context we often write $\At^\nabla_{\Xcal/S}(\Pcal)$ as a homomorphism
\begin{equation}\label{eqn:ks}
\At^\nabla_{\Xcal/S}(\Pcal)\,\,:\,\, \theta_S\,\longrightarrow\, R^1f_*\ad(\Pcal)^\nabla.
\end{equation}
In Subsection \ref{subsect:comparison} we shall define a morphism from $S$ to the so-called
character stack and then interpret this homomorphism $\At^\nabla_{\Xcal/S}(\Pcal)$ as the derivative of this morphism.

For $\Phi\,\in\, \CC[\gfrak]^G$ homogeneous of degree $n$, its evaluation on $\At^\nabla_{\Xcal/S}(\Pcal)$ now yields
\[
\Phi(\At_{\Xcal/S}^\nabla(\Pcal))\,\in\, H^0(S,\,\Omega_S^n\otimes R^nf_*f^{-1}\Ocal_S)\,
\cong \,H^0(S,\,\Omega_S^n\otimes R^n f_*\underline{\CC}_\Xcal).
\]
So this is a twisted holomorphic $n$-form on $S$, to be precise, it is
a holomorphic $n$-form with values in the local system $R^n f_*\underline{\CC}_\Xcal$. 
We thus obtain a homomorphism of graded algebras $\CC[\gfrak]^G\,\longrightarrow\,
\bigoplus_{n\ge 0} H^0(S,\,\Omega_S^n\otimes R^n f_*\underline{\CC}_\Xcal)$,
where the cup product furnishes the product structure on the right hand side.

\begin{corollary}\label{cor2:}
In this situation, the relative Chern-Weil homomorphism takes its values in the closed forms: it is a graded algebra homomorphism
\[
\CC[\gfrak]^G\,\,\longrightarrow\,\,\bigoplus_{n\ge 0} H^0(S,\,\Omega_{S,cl}^n\otimes R^n
f_*\underline{\CC}_\Xcal)
\]
whose definition commutes with base change. In particular, if $n$ is the real fiber dimension of $f$ (so that
$n$ is even) and $\Phi\,\in\, \CC[\gfrak]^G$ is homogeneous of degree $n$,
then the holomorphic $n$-form $f_! \Phi(\at_{\Xcal/S}^\nabla)$ on $S$ obtained by integration along fibers is closed,
where $\at_{\Xcal/S}^\nabla$ is the section constructed in \eqref{at1}.
\end{corollary}

\begin{proof}
The first assertion follows from Lemma \ref{lemma:Atrel} and a standard argument in the
Chern-Weil theory. The base change property is evident.
\end{proof}

\section{The universal construction}

\subsection{The character stack}\label{subsect:charstack}

We first recall the notion of a characteristic space (for which we use Sikora \cite{Si} as a general 
reference). Let $\Gamma$ be a finitely generated group, admitting $p$ generators say, so that there exists an 
epimorphism from the free group $F_p$ on $p$ generators to $\Gamma$. Then the group homomorphisms from $\Gamma$ to $G$ make up a complex space $\HHom(\Gamma,\, G)$ of finite type over $\CC$. We note that it can be realized in an 
evident manner as a closed subspace of $G^p$: every element in the kernel of the projection $F_p\,\longrightarrow\, 
\Gamma$ defines a morphism $G^p\,\longrightarrow\,G$ for which the preimage of $1\,\in\, G$ defines a closed 
subspace of $G^p$ and $\HHom(\Gamma,\, G)$ is then simply the intersection of these subspaces. It is independent 
of the epimorphism, but it need not be reduced. Its definition is covariant in $G$ and contravariant in $\Gamma$. In particular,
$\Aut(\Gamma)\times \Aut(G)$ acts on $\HHom(\Gamma,\, G)$.
Let us note in passing that the noetherian nature of complex-analytic geometry implies that the
ideal defining $\HHom(\Gamma,\, G)$ in $G^p$ is locally finitely generated and hence will involve only finitely many 
relations. So for local considerations there is no loss in generality in assuming that $\Gamma$ is finitely presented.

Composition with inner automorphisms of $G$ makes $G$ act on $\HHom(\Gamma,\, G)$ in the category of complex spaces. The kernel of this 
action clearly contains the center $Z(G)$ of $G$ and so the action of $G$ is via $\overline{G}\,:=\, G/Z(G)$. We could form the 
categorical quotient $\overline{G}\backslash\HHom(\Gamma,\, G)$ (or in case $G$ is a linear algebraic group, $\spec \CC[\HHom(\Gamma,\, 
G)]^{\overline G}$), but we prefer to treat the quotient as a stack, denoting it by $\underline{\chi}_G(\Gamma)$. We shall not recall 
the formal definition as for our purpose it suffices to know what it means to have a stack morphism from a complex space $S$ to 
$\underline{\chi}_G(\Gamma)$ (often called an $S$-valued point of $\underline{\chi}_G(\Gamma)$). In the algebraic setting (so where $S$ 
is a complex scheme of finite type), it simply assigns to each closed point of $S$ a $G$-orbit in $\HHom(\Gamma,\, G)$ in such a manner 
that this is locally (in some Grothendieck topology that has to be specified) presented as a scheme morphism from $S$ to 
$\HHom(\Gamma,\, G)$. The same description continues to be valid in the context of other categories of ringed spaces such as the 
holomorphic and $C^\infty$ categories (where $G$ is a complex Lie group respectively\ just a Lie group). An advantage of this approach is that it allows us to avoid 
imposing irrelevant additional assumptions.

It is then clear that if $f\,:\,\Xcal\,\longrightarrow\, S$ is in one of these categories, 
and $f$ is topologically locally trivial with connected fiber and simply connected base $S$, then a principal $G$-bundle 
$\Pcal$ over $\Xcal$ endowed with a connection $\nabla$ relative to $f$ determines a morphism in this category from 
$S$ to $\underline{\chi}_G(\Gamma)$, where $\Gamma$ is the fundamental group of a fiber (so given up to an inner 
automorphism). Conversely, such a morphism determines the pair $(\Pcal,\, \nabla)$ uniquely up to an obvious notion 
of isomorphism.

A \emph{regular form} on $\underline{\chi}_G(\Gamma)$ is by definition a regular form $\omega$ on the complex space 
$\HHom(\Gamma,\, G)$ that is \emph{basic} with respect to the $G$-action; this means that
$$
i_{\mathbf v} \omega\,=\, 0 \quad \text{ and }\quad i_{\mathbf v}d\omega\,=\, 0 
$$
for all $v\, \in\, \gfrak$, where ${\mathbf v}$ is the vector field on $\HHom(\Gamma,\, G)$ corresponding to
$v$ for the action of $G$ on $\HHom(\Gamma,\, G)$. If $\omega$ is closed, then the second property
is obviously automatic and hence such a form $\omega$ is basic if and only if its contraction with $\gfrak$ is basic.

Let us also note that the obvious 
action of $\Aut(\Gamma)$ on $\HHom(\Gamma,\, G)$ gives rise to an action of $\Aut(\Gamma)$ on the space of regular forms on 
$\underline{\chi}_G(\Gamma)$ and that (as one would expect) this action factors through the outer automorphism group 
$\Out(\Gamma)$.

We have a similar way of describing the Zariski tangent space of $\underline{\chi}_G(\Gamma)$ at a point, represented 
by $\rho\,\in\, \HHom(\Gamma,\, G)$, say. Recall that the space of $1$-cycles $Z^1(\Gamma,\, \Ad\rho)$ consists of 
maps $\Gamma\,\longrightarrow\, \gfrak$ that are completely given by their values on a set of generators. For 
$Z^1(F_p,\, \Ad\rho)$ these values can be arbitrary so that $Z^1(\Gamma,\, \Ad\rho)$ becomes a subspace of 
$$Z^1(F_p,\, \gfrak)\,\cong\, Z^1(F_p)\otimes\gfrak\,\cong\,\gfrak^{\oplus p}$$ (which is independent of $\rho$). 
When we regard $\rho$ as an element of the algebraic group $G^p$, then the right translation with $\rho$ identifies 
$Z^1(\Gamma,\, \Ad\rho)$ with the Zariski tangent space of $\HHom(\Gamma,\, G)$ at $\rho$ \cite[Theorem\, 35]{Si}. 
To be precise: if $t\,\longmapsto \,\rho_t$ is a curvelet in $\HHom(\Gamma,\, G)$ given up to first order with $\rho_0=\rho$ and 
we write $\rho_t\equiv e^{t\sigma}\rho_0\pmod{t^2}$, then the condition $\rho_t(\gamma_1\gamma_2)\,=\,
\rho_t(\gamma_1)\rho_t(\gamma_2)$ boils down to $\sigma (\gamma_1\gamma_2)\,=\,
\sigma (\gamma_1) + \rho(\gamma_1)\sigma (\gamma_2)\rho(\gamma_1)^{-1}$, which indeed amounts to 
$\sigma\in Z^1(\Gamma,\, \gfrak)$.
Under this map, the space of $1$-boundaries $B^1(\Gamma,\, \Ad\rho)$ is mapped on the tangent space of the $G$-orbit 
of $\rho$ and therefore the group cohomology space $H^1(\Gamma,\,\Ad\rho)$ can be understood as the Zariski tangent 
space of $\underline\chi_G(\Gamma)$ at $\underline\rho$. There is a similar notion for the Zariski tangent space of an 
$S$-valued point (in a manner that is straightforward to spell out); this then appears as a coherent 
$\Ocal_S$-module over $S$ that we shall denote by $\theta_{\underline\chi_G(\Gamma), S}$.

The regular forms on $\underline{\chi}_G(\Gamma)$ of interest here take values in $H^\pt(\Gamma, \,\CC)$ and are 
defined by the $G$-invariants in $\CC[\gfrak]$: any $\Phi\,\in\, \CC[\gfrak]^G$ that is homogeneous of degree $n$ 
combines with the cup product to give a linear map
\[
\eta(\rho,\Phi) :\wedge^n H^1(\Gamma,\, \Ad\rho) \,\,\longrightarrow\,\, H^n(\Gamma,\, (\Ad\rho)^{\otimes n})
\,\,\longrightarrow\,\, H^n(\Gamma,\, \CC).
\] 
We can now establish part of our main Theorem \ref{thm:main}.

\begin{proposition}\label{prop:eta}
The linear maps $\eta(\rho,\Phi)$ define a $H^n(\Gamma,\, \CC)$-valued $n$-form $\eta_\Phi$ on $\HHom(\Gamma,\, G)$ that is basic relative to the $G$-action.
We thus obtain a graded $\CC$-algebra homomorphism
\[
\eta\,\,: \,\,\CC[\gfrak]^G\,\,\longrightarrow\,\, \bigoplus_{n\ge 0} \big(H^0(\underline{\chi}_G(\Gamma),\,\,
\Omega^n_{\underline{\chi}_G(\Gamma)})\otimes H^n(\Gamma,\, {\mathbb C})\big)^{\Aut(\Gamma)},
\]
where the degree of the summand indexed by $n$ is $2n$.
\end{proposition}

\begin{proof}
Consider the natural linear map 
\[
\wedge^n(Z^1(F_p)\otimes\gfrak)\,\,\longrightarrow\, \,\wedge^n Z^1(F_p)\otimes \sym^n\gfrak\,\,\xrightarrow{\,1\otimes\Phi\,\,}
\,\, \wedge^n Z^1(F_p,\, \CC).
\]
(It descends to a map $\wedge^n H^1(F_p,\, \Ad\rho)\,\longrightarrow\, H^n(F_p,\, {\mathbb C})$, but this is of little interest 
here, as $H^n(F_p,\,{\mathbb C})\,=\,0$ for all $n\,>\,1$.) We identify $Z^1(F_p)\otimes \gfrak\,\cong \,\gfrak^{\oplus p}$ with 
the Lie algebra of $G^p\,\cong\,\HHom(F_p,\, G)$ and then use the left translation by $G^p$ to make this an $n$-form 
$\widetilde{\eta}_\Phi$ on $\HHom(F_p,\, G)\,=\,G^p$ with values in $\wedge^n Z^1(F_p,\,\CC)$. 

This is, besides being 
invariant under the left $G^p$-action, also invariant under simultaneous conjugation by $G$ (for $\Phi$ is 
$G$-invariant). The same is true for its exterior derivative and so $\widetilde{\eta}_\Phi$ is basic.

The restriction of $\widetilde{\eta}_\Phi$ to the closed subspace $\HHom(\Gamma,\, G)$ of $\HHom(F_p,\, G)$ is of course 
still basic. But it will now take its values in the subspace $\wedge^n Z^1(\Gamma)\,\subset\, \wedge^n Z^1(F_p,\,\CC)$ and 
$\eta_\Phi$ then appears after composition with the natural map $$\wedge^n Z^1(\Gamma,\,\CC)\,\longrightarrow\, \wedge^n 
H^1(\Gamma,\,\CC)\,\longrightarrow\, H^n(\Gamma,\,\CC ).$$ This proves that $\eta_\Phi$ is basic.

The naturality of the construction implies that $\eta_\Phi$ is $\Aut(\Gamma)$-invariant.
\end{proof}

\begin{remark}\label{rem:}
We already noted that $\underline\chi_G(\Gamma)$ can be regarded as the moduli stack of $G$-local systems on a 
reasonable space with fundamental group isomorphic to $\Gamma$. In case that space is a connected compact 
$C^\infty$-manifold $M$, there is a geometric interpretation of $H^1(\Gamma,\, \Ad\rho)$ in terms of flat connections 
that is rooted in Yang-Mills theory (as developed by Atiyah and Bott in \cite{AB}) and which in the present context was first 
explicated by Goldman \cite{Go}. For this, assume that we are given a universal cover $\widetilde M\,\longrightarrow\, M$ with group of deck 
transformations $\Gamma$ (acting on the left). Every $\rho\,\in \,\HHom(\Gamma,\, G)$ gives rise to a principal 
$G$-bundle $P_\rho$ on $M$ with a flat connection; it is obtained as the quotient of $\widetilde M\times G$ by the 
diagonal action of $\Gamma$ (acting on $G$ by left translations via $\rho$). But if we vary $\rho$, then the underlying 
$C^\infty$ principal bundle does not really change if we ignore the connection. In other words, if we fix a 
$\rho\,\in\,\HHom(\Gamma, \,G)$ and write $(P,\, \nabla)$ for $P_{\rho}$, then a variation of $\rho$ can be implemented by a 
variation of $\nabla$ as a flat connection. Any $C^\infty$-connection $\nabla'$ on $P$ is of the form $\nabla+A$ 
with $A$ being a section of $\Ecal_M^1\otimes \ad(P)$ and the connection $\nabla+A$ is flat if and only if $d_\nabla A+A\wedge A=0$ 
(the exterior derivative being taken with respect to the connection $\nabla$). For the first order deformations we 
can ignore the quadratic term $A\wedge A$ so that this becomes $d_\nabla A\,=\,0$. Those that come from an infinitesimal 
automorphism (and hence `do not count') are in the $d_\nabla$-image of a smooth section of $\ad(P)$ and so this 
realizes the space of genuine first order deformations of $\nabla$ with the de Rham cohomology space $H^1(M,\,\ad(P, 
\nabla))\,=\,H^1(M,\,\ad(P_\rho))$. The standard isomorphism
\begin{equation}\label{eqn:compare}
H^1(\Gamma,\, \Ad\rho)\xrightarrow{\,\,\cong\,\,}\, H^1(M, \,\ad(P_{\rho}))\end{equation}
thus acquires a geometric interpretation, namely as 
the tangent map at $\rho$ of the map which assigns to any $[\rho']\,\in\, \underline\chi_G(\Gamma)$, the isomorphism type of 
$P_{\rho'}$. We will encounter an analogue of it for $S$-valued points of $\underline\chi_{G}(\Gamma)$ in the next 
subsection.
\end{remark}

\subsection{A comparison}\label{subsect:comparison}

We return to the situation of Subsection \ref{subsect:relflat}. Since the issue that we 
will address is local on $S$, we assume that $S$ is a contractible Stein manifold, that 
$\Xcal$ is connected and that $f$ admits a holomorphic section 
$x\,:\,S\,\longrightarrow\, \Xcal$ such that $x^*\Pcal$ is trivial, i.e., 
$x^*\Pcal$ admits a holomorphic section $\widetilde x$. Let $p\,:\, 
\widetilde\Xcal\,\longrightarrow\, \Xcal$ be the universal cover defined by $x$ with 
group of deck transformations $\Gamma\,:=\,\pi_1(\widetilde\Xcal,\, x)$. Since $S$ is 
contractible, the map $p$ defines a universal cover of every fiber of $f$. The pull-back 
$p^*\Pcal$ comes with a connection relative to $f\circ p$ that is flat. Since the fibers 
of $f\circ p$ are simply connected, the global sections of $p^*\Pcal$ that are relatively 
flat identify this pull-back with $\widetilde\Xcal\times_S x^*\Pcal$. We use the 
section $\widetilde x$ of $ x^*\Pcal$ to identify this even with $\widetilde\Xcal 
\times G$. Via this isomorphism the $\Gamma$-action on $p^*\Pcal$ then becomes one on 
$\widetilde\Xcal \times G$. The latter action is given by a holomorphic holonomy map
$$\rho\, :\, S\,\longrightarrow \,
\HHom(\Gamma,\, G)\, , \ \ s\, \longmapsto\, \rho_s$$
in the sense that $\gamma\,\in\, \Gamma$ takes $(\widetilde{x},\, g)$ to $(\gamma.\widetilde x,\, \rho_{f(p(\widetilde x))}(\gamma)g)$.
A different 
choice of $ x$ or $\widetilde x$ alters $\rho$ by conjugation with a holomorphic map $S\,\longrightarrow\, G$, so that the 
resulting holomorphic map $\underline\rho\,:\, S\,\longrightarrow\, \underline\chi_{G}(\Gamma)$ is intrinsic to the situation.

The $\Gamma$-cover $\widetilde\Xcal$ of $\Xcal$ defines a natural homotopy class $\kappa$ of maps from $\Xcal $ to a
classifying space for $\Gamma$. It has its own characteristic algebra, given by the natural map:
\[
\kappa^* \,:\, H^\pt(\Gamma)\,\longrightarrow\, H^\pt (\Xcal).
\]
Proposition \ref{prop:comparison} and Corollary \ref{cor:closed} below imply the remaining assertions of Theorem \ref{thm:main}.

\begin{proposition}\label{prop:comparison}
If we regard $\kappa^*$ as a homomorphism of trivial local systems on $S$:
\[
\kappa^*\,:\, \underline{H^\pt(\Gamma)}_S\,\,\longrightarrow\,\, R^\pt 
f_*\underline{\ZZ}_\Xcal ,
\]
then $1\otimes\kappa^*$ takes $\rho^*\eta_\Phi\,\in\, H^0(S,\,\Omega_S^{n})\otimes H^{n}(\Gamma,\, 
\CC)$ to $\Phi(\At_{\Xcal/S}^\nabla(\Pcal))$, where $\At_{\Xcal/S}^\nabla(\Pcal)$ and $\eta$ are as in Lemma \ref{lemma:Atrel}
and Proposition \ref{prop:eta} respectively.
\end{proposition}

\begin{proof}
For $s\,\in\, S$, put $X_s\,:=\,f^{-1}(s)$ and let $(P_s/X_s,\, \nabla_s)$ be the restriction $(\Pcal,\,\nabla)$ to 
$X_s$. We have the Kodaira-Spencer homomorphism 
$$\At^\nabla_{\Xcal/S}(\Pcal)\,:\, \theta_S \,\longrightarrow\, R^1f_*\ad(\Pcal)^\nabla$$
in \eqref{eqn:ks}, whereas the variation of the local system defines a similar 
homomorphism from $\theta_S$ its Zariski tangent bundle $\theta_{\underline\chi_{G}(\Gamma),S}$ when considered as a 
$S$-valued point (for a closed point $s$ this would be $H^1(\Gamma, \,\Ad\rho_s)$). As noted above, the isomorphism 
$H^1(\Gamma,\, \Ad \rho_s)\,\xrightarrow{\,\cong\,}\, H^1(X_s,\, \ad(P_s))$ has an interpretation as an isomorphism 
of tangent spaces. Over $S$ this gives rise to an isomorphism 
$\theta_{\underline\chi_{G}(\Gamma),S}\,\longrightarrow\, R^1f_*\ad(\Pcal)^\nabla$ and what remains to show is that 
this is compatible with the homomorphisms from $\theta_S$ to its source and target. The proof of this is left to the 
reader.
\end{proof}

\begin{corollary}\label{cor:closed}
The map $\eta$ in Proposition \ref{prop:eta} takes values in the \emph{closed} $H^\pt(\Gamma,\,
\CC)$-valued forms on $\underline{\chi}_G(\Gamma)$.
\end{corollary}

\begin{proof}
We observed that the characteristic stack $\underline\chi_G(\gamma)$ is locally realized by a finitely presented group that 
has $\Gamma$ as a quotient and so there is no loss in generality in assuming that $\Gamma$ is already of that type. 
If $\Gamma$ appears as the fundamental group of compact complex manifold, then Proposition \ref{prop:comparison} gives the 
corollary for such a $\Gamma$. At this point we can finish the proof by invoking a theorem of Taubes, 
\cite[Theorem\, 1.1]{Ta}, which implies that any finitely presented group is isomorphic to the fundamental group a 
compact complex manifold of dimension three (which here appears as the twistor space of an
anti-self-dual Riemannian $4$-manifold). Instead 
of relying on such a deep theorem, one might alternatively observe that Proposition \ref{prop:comparison} and its 
proof are also valid in case $\Xcal\,\longrightarrow\, S$ is a trivial family with fiber a compact complex
manifold $M$ with 
boundary. In that case it is clear that any finitely presented group can be realized as the fundamental group of 
such a manifold: just take a finite cell complex $K$ with that fundamental group, embed $K$ in $\CC^3$ and take for 
$M$ be a closed regular neighborhood of its image.
\end{proof}

\begin{remark}\label{rem3:}
It would of course be more satisfying to have a direct proof of Corollary \ref{cor:closed}. For this it is perhaps useful to 
point out that the corollary does not hold without tensoring with $H^\pt(\Gamma,\,\CC)$ (which can be zero). For example, if
$\Gamma\,=\,F_p$ with $p\,\ge\, 2$, then the image of the Killing form on $\slin(2,\CC)$ in $\HHom(\Gamma,\, \PSL_2(\CC))$ will
not be closed, but after tensoring with $H^2(F_p,\,{\mathbb C})\,=\,0$ it is of course trivially so.
\end{remark}

%%%%%%%%%%%%%%%%%%%%%%%%%%%%%%%%%%%%%%%%%%%%%%%%%%%%%%%%%%%%%%%%%

\end{document}